\documentclass[12pt]{amsart}
\usepackage{amsmath,latexsym,amscd,amsbsy,amssymb,amsfonts,amsthm,fleqn,leqno,
euscript, graphicx, texdraw, pb-diagram}
\usepackage[matrix,arrow,curve]{xy}
\numberwithin{equation}{section}
\newtheorem{thm}{Theorem}[section]
\newtheorem{pro}[thm]{Proposition}
\newtheorem{lem}[thm]{Lemma}

\def\con{\subset}
\def\from{\colon}
\def\CL#1{\overline{#1}}

\def\leukfrac#1/#2{\leavevmode
               \kern.1em
                \raise.9ex\hbox{\the\scriptfont0 ${}_#1$}
                \hskip -1pt\kern-.1em
                /\kern-.15em\lower.10ex\hbox{\the\scriptfont0 ${}_#2$}}

\theoremstyle{definition}

\theoremstyle{remark}
\newtheorem{claim}{Claim}

\newcommand\sphere{\mathbb{S}}
\newcommand\interval{\mathbb{I}}

\makeatletter
\def\bd{\mathop{\operator@font bd}\nolimits}
\def\diam{\mathop{\operator@font diam}\nolimits}
\makeatother

\begin{document}
%%%%%%%%%%% Begin Topmatter %%%%%%%%%%%%%%%%%

\title[Homogeneous continua that are not separated by arcs]
{Homogeneous continua that are not separated by arcs}

\author{J. van Mill}
\address{KdV Institute for Mathematics\\
University of Amsterdam\\
Science Park 105-107\\
P.O. Box 94248\\
1090 GE Amsterdam, The Netherlands}
\email{j.vanMill@uva.nl}

\author{V. Valov}
\address{Department of Computer Science and Mathematics,
Nipissing University, 100 College Drive, P.O. Box 5002, North Bay,
ON, P1B 8L7, Canada} \email{veskov@nipissingu.ca}

\thanks{The first author is pleased to thank the Nipissing University for generous
hospitality and support. The second author was partially supported by NSERC
Grant 261914-13.}


\begin{abstract}
We prove that if $X$ is a strongly locally homogeneous and locally compact separable metric space and $G$ is a region in $X$ with $\dim G=2$, then $G$ is not separated by any arc in $G$.
\end{abstract}
\maketitle\markboth{}{Homogeneous continua}
%%%%%%%%%% End topmatter %%%%%%%%%%%%%%%%%%%%%

%%%%%%%%%%%%%%%%%%%%%%%%%%%%%%%%%%%%%%%%%%%%%%%%%%%%%%%%%%%%%%%%
%%%%%%%%%%%%%%%%%%%%%%%%%%%%%%%%%%%%%%%%%%%%%%%%%%%%%%%%%%%%%%%%

%%%%%%%%%%%%%%%%%% TABLE OF CONTENT %%%%%%%%%%%%%%%%%%%%%%%%%%%%%%%%%%

%\tableofcontents

%%%%%%%%%%%%%%%%%%%%%%%%%%%%%%%%%%%%%%%%%%%%%%%%%%%%%%%%%%%%%%%%%%%
%%%%%%%%%%%%%%%%%%%%%%%%%%%%%%%%%%%%%%%%%%%%%%%%%%%%%%%%%%%%%%%%%%%%%%

\section{Introduction}
By a \emph{space} we mean a separable metric space.  Kallipoliti and Papasoglu~\cite{kp} proved that any locally connected, simply connected, homogeneous metric continuum can not be separated by arcs, and asked if this is true without the assumption of simply connectedness. A partial answer of this question was provided in \cite{vv} for homogeneous metric continua of dimension two having a non-trivial second integral \v{C}ech cohomology group. In the present paper we prove the following partial answer to Kallipoliti and Papasoglu's question.

\begin{thm}\label{eerstestelling}
Let $X$ be a locally compact strongly locally homogeneous space  and $G$ be a region in $X$ with $\dim G=n\geq 2$. Then $G$ is not separated by any arc $J\subset G$.
\end{thm}

%\begin{cor}
%Let $X$ be a locally compact strongly homogeneous connected space of dimension $\geq 2$ and $G$ be a region in $X$. Then there is no arc $J\subset G$  and a representation $G=\bigcup_{i\geq 2}F_i$, where $F_i$ are closed in $G$ and $F_i\cap F_j\subset J$ for each $i\neq j$.
%\end{cor}
Recall that a space is strongly locally homogeneous if every point $x\in X$ has a local basis of open sets $U$ such that for every $y,z\in U$ there is a homeomorphism $h$ on $X$ with $h(y)=z$ and $h$ is identity on $X\setminus U$. Obviously, every open subset of a strongly locally homogeneous space is also strongly locally homogeneous. Since
strongly locally homogeneous connected spaces are homogeneous, any region $G$ satisfying the hypotheses of Theorem 1.1 should be homogeneous.  We claim that it is locally connected as well.
Indeed, since any strongly homogeneous Polish space is countable dense homogeneous \cite{bp} and a locally compact countable dense homogeneous connected space is locally connected \cite{f}, we have that any region $G$ from Theorem 1.1 is locally connected. (There is also a simple direct proof of this fact.)
According to \cite{kru}, every region of homogeneous locally compact space of dimension $n\geq 1$ can not be separated by a closed set of dimension $\leq n-2$. So, Theorem 1.1 is interesting only for regions $G$ of dimension two.

%Everywhere below we consider reduced in dimension zero \v{C}ech cohomology groups $H^n(X)$ with coefficients from the integers $\mathbb Z$. Recall that any cohomology group $H^n(X)$ is isomorphic to the group $[X,K(\mathbb Z,n)]$ of pointed homotopy classes of maps from $X$ to $K(\mathbb Z,n)$, where $K(\mathbb Z,n)$ is the Eilenberg-MacLane space of type $(\mathbb Z,n)$, see \cite{sp}.
%It is also well known that the circle group $\mathbb S^1$ is a space of type $(\mathbb Z,1)$.

%Suppose $(K,A)$ is a pair of closed subsets of a space $X$ with $A\subset K$. We say that $K$ is
%an {\em $n$-cohomology membrane spanned on $A$ for an element $\gamma\in H^n(A)$} if $\gamma$ is not extendable over $K$,
%but it is extendable over every proper closed subset of $K$ containing $A$. Here, $\gamma\in H^n(A)$ is not extendable over $K$ means that
%$\gamma$ is not contained in the image $j_{K,A}^n\big(H^{n}(K)\big)$, where $j_{K,A}^n:H^{n}(K)\to H^{n}(A)$ is the
%homomorphism generated by the inclusion $A\hookrightarrow K$.

\section{Some preliminary results}

\begin{lem}\label{retraction}
Let $A$ be a closed nowhere dense subset of $X$ such that $\dim X\setminus A = 0$. Then there is a retraction $r\from X\to A$ such that $r(X\setminus A)$ is countable.
\end{lem}

\begin{proof}
The technique is similar to that in \cite{KnasterReichbach}. In brief, one constructs a cover $\mathcal{V}= \{V_n : n\in\mathbb N\}$ by disjoint nonempty clopen subsets of $X$ such that
\begin{enumerate}
\item $\diam V_n < d(V_n,A)$ for each $n$,
\item there is a sequence $\{a_n: n\in \mathbb N\}$ in $A$ such that $$\lim_{n\to\infty} d(a_n,V_n) = 0.$$
\end{enumerate}
Then define $r\from X\to A$ as follows: $r(a)=a$ for every $a$ and $r(V_n) = \{a_n\}$ for every $n$. It is easy to check that $r$ is as required.
\end{proof}

If $J$ is an arc and $p,q\in J$, then $(p,q)$ and $[p,q]$ denote, respectively, the open and closed subintervals in $J$ with endpoints $p,q$.
\begin{pro}\label{neighborhood}
Let $J=[a,b]$ be an arc in a space $X$ which is everywhere 2-dimensional. Then $b$ has arbitrarily small open neighborhoods $U$ such that $\bd(U)$ is at most 1-dimensional and intersects $J$ in exactly one point.
\end{pro}

\begin{proof}
Fix $\varepsilon > 0$ and let $U$ be an open neighborhood of $b$ in $X$ such that $\diam \CL{U}< \varepsilon$ and $\dim \bd U\le 1$. We may assume without loss of generality that $J\setminus U \not=\emptyset$ and $J\cap U$ is uncountable. Put $Y = J \cup \CL{U}$. Moreover, put $A = J\cup \bd U$, $B= (J\setminus U)\cup \bd U$ and
$C=(J\cap\overline U)\cup\bd U$, respectively.

Let $D$ be a zero-dimensional dense subset of $U$ such that $\dim U\setminus D=1$. Since $\dim J = 1$, we may clearly assume that $D\cap J = \emptyset$.

Because $C$ is a closed nowhere dense subset of $C\cup D$, there is a retraction $r_1:C\cup D\to C$ such that $r_1(D)$ is countable (Lemma~\ref{retraction}).
Let $r\from A\cup D\to A$ be defined by $r(x)=r_1(x)$ if $x\in C\cup D$ and $r(x)=x$ if $x\not\in C\cup D$. Obviously $r$ is a retraction such that $r(D)$ is countable.  Pick an arbitrary $s\in U\cap J$ such that $s\not= b$, $[s,b]\con U$ and $s\not\in r(D)$. Choose also two points $s_1, s_2\in J\cap U$ different from $s$ and $b$ such that $s\in (s_1,s_2)$, and let
$V_1=A\setminus [s_1,b]$ and $V_2=(s_2,b]$. Obviously $V_1$ and $V_2$ are open subsets of $A$ containing $B$ and $\{b\}$, respectively. Moreover,
$\overline V_1=A\setminus (s_1,b]$ and $\overline V_2=[s_2,b]$.

\begin{claim}
$\{s\}$ is a partition in $A$ between $\overline V_1$ and $\overline V_2$.
\end{claim}

Indeed, put $P= [s,b]$ and $Q= [a,s]\cup \bd U$. Then $P$ and $Q$ are closed subsets of $A$ such that $P\cup Q = A$, $\overline V_2\subset P$, $\overline V_1\con Q$ and $P\cap Q=\{s\}$.

\medskip

\begin{claim}
$\{s\}$ is a partition in $A\cup D$ between $r^{-1}(\overline V_1)$ and $r^{-1}(\overline V_2)$.
\end{claim}

Since $r^{-1}(s) = \{s\}$, this is a direct consequence of Claim~1.

\medskip

By \cite[Lemma 3.1.4]{vm:book:twee}, there is a partition $S$ between $\{b\}$ and $B$ in $Y$ such that $S\cap (A\cup D) \con \{s\}$. If $s\not\in S$, then $S\cup \{s\}$ is also a partition between $\{b\}$ and $B$ in $Y$, hence we may assume without loss of generality that $s\in S$. But then $S\cap J = \{s\}$.
Write $Y\setminus S$ as $E\cup F$, where $E$ and $F$ are disjoint relatively open subsets of $Y$ such that $b\in E$ and $B\con F$.

\begin{claim}
$E\con U$.
\end{claim}

Indeed, since $E\cap B = E\cap {\big (}(J\setminus U)\cup \bd U{\big )} =\emptyset$, this is clear.

\medskip

Since $E$ is open in $U$ and $U$ is open in $X$ we have that $E$ is open in $X$. Moreover, $\diam E < \varepsilon$. Also, $E\cup S$ is closed in $Y$ and hence in $X$. As a consequence $\bd E\con S$. Since $S\con U\setminus D$, we have $\dim S\le 1$, as required.
\end{proof}
It will be convenient to use additive notation for the topological group $\sphere^1$.

The following result can be proved by tools from algebraic topology. For the convenience of the reader, we include a simple direct proof.

\begin{pro}\label{promembrane}
Let $X$ be a space and let $A$ be a closed subspace of it. Moreover, let $\gamma\from A\to \sphere^1$ be continuous.
Suppose that there are closed subsets $P_1,P_2$ of $X$ satisfying the following conditions:
\begin{itemize}
\item $P_1\cup P_2=X$ and if $C=P_1\cap P_2$ then $C\cap A$ is a singleton, say $c$;
\item $\gamma| P_i\cap A$ is extendable over $P_i$ for each $i=1,2$, but $\gamma$ is not extendable over $X$.
\end{itemize}
Then there is a continuous function $\beta\from C\to \mathbb S^1$ such that $\beta (c) = 0$ and $\beta$ is not nullhomotopic.
\end{pro}

\begin{proof}
Let $\alpha_i\from P_i\to \mathbb S^1$ for $i=1,2$ be a continuous extension of $\gamma| P_i\cap A$. Define $\beta\from C\to \mathbb S^1$ by $\beta(x) = \alpha_1(x)-\alpha_2(x)$ $(x\in C)$. Then, clearly, $\beta (c) = 0$. We claim that $\beta$ is as required, and argue by contradiction. Assume that $\beta$ is nullhomotopic. Let $H\from C\times \interval\to \mathbb S^1$ be a homotopy such that $H_0 \equiv 0$ and $H_1=\beta$. Define $S\from C\times \interval\to \mathbb S^1$ by $S(x,t) = H(x,t) - H(c,t)$. Then $S_0 \equiv 0$, $S_1 = \beta$ and $S(c,t)=0$ for every $t$. Define a homotopy $T\from (C\cup (P_2\cap A))\times\interval \to \mathbb S^1$ by
$$
    T(x,t) = \begin{cases}
              S(x,t) & (x\in C, t\in \interval), \\
              0      & (x\in P_2\cap A, t\in\interval).
            \end{cases}
$$
Then $T_0 \equiv 0$ and hence can be extended to the constant function with value 0 on $P_2$. By the Borsuk Homotopy Extension Theorem~\cite[1.4.2]{vm:book:twee}, the function $T_1$ can be extended to a continuous function $\delta\from P_2\to \sphere^1$. Now define $\varepsilon\from X\to \sphere^1$ as follows:
$$
    \varepsilon | P_1 = \alpha_1, \quad \varepsilon | P_2 = \delta + \alpha_2.
$$
If $x\in C$, then $\varepsilon|P_1 (x) = \alpha_1(x)$ and $\varepsilon|P_2 (x)= \delta(x) + \alpha_2(x) = S_1(x) +\alpha_2(x) = \beta(x) + \alpha_2(x) = \alpha_1(x)$. Hence $\varepsilon$ is well defined and continuous. Also observe that if $x\in P_2\cap A$, then
$$
    \varepsilon(x) = 0 + \alpha_2(x) = \alpha_2(x).
$$
Hence $\varepsilon$ extends $\gamma$, which is a contradiction.
\end{proof}

\section{Proof of Theorem 1.1}

 Throughout, let $X$ be a locally compact and strongly locally homogeneous space, and $G$ be a region in $X$ of dimension $2$. Suppose $G$ is separated by an arc $J=[a,b]\subset G$.
 Recall that $G$ is homogeneous and locally connected (see \S 1).
 Write $G\setminus J$ as $G_1\cup G_2$, where $G_1$ and $G_2$ are disjoint nonempty open subsets of $G$. Everywhere below $\overline K$ denotes the closure of $K$ in $G$ for any set $K\subset G$.
% We also denote by $G_X$ the group of all homeomorphisms from $X$ onto $X$.

%Clearly, $J$ is uncountable. For if $J$ would be countable, it would be 0-dimensional which would violate Krupski's result from~\cite{Krupski90}.

%We think of $X$ as a subspace of its one-point compactification $\alpha X$. We let $d$ denote an admissible metric on $\alpha X$.

%Write $X\setminus J$ as $G_1\cup G_2$, where $G_1$ and $G_2$ are disjoint nonempty open subsets of $X$.

We say that a space $Y$ has no local cut
points if no connected open subset $U \con Y$ has a cut point.

\begin{lem}\label{eerstelem}
$G$ has no local cutpoints.
\end{lem}

\begin{proof}
By Kruspki~\cite[Theorem 2.1]{kru} it follows that every nonempty open connected subset $U$ of $G$ is a Cantor manifold of dimension~2. Hence $U$ cannot be separated by a zero-dimensional closed set.
\end{proof}

A space $X$ is \emph{crowded} if it has no isolated points.

\begin{lem}\label{crowded}
The set $S= \CL{G}_1\cap \CL{G}_2$ is a 1-dimensional closed and crowded subspace of $J$ which separates $G$.
\end{lem}
\begin{proof}
Assume first that $J\setminus (\CL{G}_1\cup \CL{G}_2) \not=\emptyset$. Then $G$ is somewhere at most 1-dimensional. Hence $G$ is at most 1-dimensional at every point by homogeneity. But this contradicts $G$ being 2-dimensional.

Hence $J\con \CL{G}_1\cup \CL{G}_2$ and so $G=\CL{G}_1\cup \CL{G}_2$.
If $S$ is empty, then $G$ is covered by the disjoint nonempty closed sets $\CL{G}_1$ and $\CL{G}_2$ which contradicts the connectivity of $G$.

Now assume that $x$ is an isolated point of $S$. Let $U$ be an open connected neighborhood of $x$ in $G$ such that $U\cap S = \{x\}$. Then $x$ is a cutpoint of $U$. But this contradicts Lemma~\ref{eerstelem}.

We conclude that $S$ separates $G$ and consequently has to be 1-dimensional by Krupski~\cite{kru}.
\end{proof}

Let $s$ be the maximum of $S$ (as a subset of $[a,b]$). Then $J_s=[a,s]$ also separates $G$ and $G\setminus J_s$ is the union of the disjoint open sets $G_1'$ and $G_2'$, where $G_i'=\overline G_i\setminus J_s$. Moreover, $s\in\overline G_1'\cap\overline G_2'$.
Hence, we can assume without loss of generality that $b\in\CL{G}_1\cap\CL{G}_2$.

\begin{lem}\label{kuzminov}
There is an open neighborhood $U\subset G$ of $b$ having compact closure and a compact set $F\subset G$ such that for every open neighborhood $V$ of $b$ with $\CL{V}\con U$ there exist a compact set  $M_U\subset\CL U$ and a continuous function $f\from \bd_F(U\cap F) \to \sphere^1$ such that:
\begin{itemize}
\item[(1)] $b\in U\cap F$;
\item[(2)] $M_U$ is everywhere $2$-dimensional and $M_U\cap V\neq\varnothing$;
\item[(3)] $\dim \bd U \le 1$ and $J\cap\bd U$ is a point;
\item[(4)] $f$ is not extendable over $\bd_F(U\cap F)\cup M_U$, but it is extendable over $\bd_F(U\cap F)\cup P$ for every proper closed set $P$ of $M_U$.
\end{itemize}
\end{lem}

\begin{proof}
Choose a compact neighborhood $O_b$ of $b$ in $G$. Since every neighborhood of $b$ is of dimension $2$,
there is a compact subset $Y\subset O_b$, a closed set $A\subset Y$ and a continuous function $g\from A \to \sphere^1$ not extendable over $Y$. Let $F$ be a minimal closed subset
of $Y$ containing $A$ such that $g$ is not extendable over $F$. Then for every open subset $W$ of $F\setminus A$ with $\overline W\cap A=\varnothing$ there is a function $f_W:F\setminus W\to\mathbb S^1$ extending $g$ such that $f_W$ can not be extended to a continuous function $\bar f_W:F\to\mathbb S^1$. This means that
$f_W|\bd_FW$ is not extendable over $\overline W$. Consequently, $F\setminus A$ is everywhere two-dimensional.
We can assume by homogeneity that $b\in F\setminus A$. Indeed, by Effros' theorem \cite{e}, we take $O_b$ so small that for every point $x\in O_b$ there is a homeomorphism $h$ on $G$ with $h(b)=x$ and $O_b\subset h(G)$. Then, consider the set $h(G)$ instead of $G$.

By Proposition~\ref{neighborhood}, there are an open neighborhood $U$ of $b$ whose closure in $G$ is a compact  and a point $c\in (a,b)$ such that $\bd U \cap J = \{c\}$, $\dim\bd U\leq 1$ and $\overline U\cap A=\varnothing$. Suppose $V$ is an open neighborhood of $b$ such that $\overline V\subset U$, and consider
a continuous function $f_V:F\setminus V\to\mathbb S^1$ extending $g$ which is not extendable over $F$. Let $f=f_V|\bd_F(U\cap F)$.
%Since $\bd_F (U\cap F) \con \bd U$ and $\dim \bd U \le 1$, the function $f_1$ can be extended to a continuous function $f\from \bd U\to \sphere^1$.
Clearly, $f$ cannot be extended to a continuous function $\bar f\from \CL{U\cap F}\to \sphere^1$, but $f$ can be extended to a continuous function from
$(\overline{U\cap F})\setminus V$ into $\mathbb S^1$. Let $M_U$ be a minimal closed subset of $\overline{U\cap F}$
with the property that $f$ cannot be extended to a continuous function
$\widetilde f:\bd_F(U\cap F)\cup M_U\to\mathbb S^1$. The minimality of $M_U$ implies that $f$ is extendable over $\bd_F(U\cap F)\cup P$ for any any closed set
$P\subsetneqq M_U$.
Because $f$ is extendable over $(\overline{U\cap F})\setminus V$, $M_U\cap V\neq\varnothing$. It is clear that $M_U$ is a continuum.

Assume that $O$ is a nonempty open subset of $M_U$ such that $\dim O \le 1$. Taking a smaller open subset of $O$, we may assume that $\dim\overline O \le 1$. There are two possibilities, either $O\con \bd_F(U\cap F)$ or $O\setminus \bd_F(U\cap F)\neq\varnothing$.
If $O\con \bd_F(U\cap F)$, $M_U\setminus O$ is a proper closed subset of $M_U$ having the same properties as $M_U$, which contradicts minimality. If
$O'=O\setminus \bd_F(U\cap F)\neq\varnothing$, then $P=M_U\setminus O'$ is a proper closed subset of $M_U$. So, there is an extension $f_1:\bd_F(U\cap F)\cup P\to\mathbb S^1$ of $f$. Since $\dim\overline {O'}\leq 1$, we can extend $f_1$ over $\bd_F(U\cap F)\cup M_U$, a contradiction. Therefore, $M_U$ is everywhere $2$-dimensional.
\end{proof}

Now, we can complete the proof of Theorem 1.1. Choose open neighborhoods $U$ and $V$ of $b$, closed sets $F\subset G$ and $M_U\subset\CL{U\cap F}$ and a continuous function $f\from \bd_F(U\cap F) \to \sphere^1$ satisfying
 the conditions $(1) - (4)$ from Lemma \ref{kuzminov}. Let also $J\cap\bd U=\{c\}$ and $C=[c,b]$. We can also assume that $V$ satisfies the additional property that for every two points $p,q\in V$ there is a homeomorphism $\varphi$ of $G$ supported on $V$ with $\varphi(p)=q$.
  We may consequently assume without loss of generality that $b\in M_U$. Indeed, if $b\not\in M_U$  we take a point $x\in M_U\cap V$ and a homeomorphism $\varphi$ of $G$ supported on $V$ such that $\varphi(x)=b$. Then the set $\varphi(M_U)$ satisfies all condition from Lemma \ref{kuzminov} and contains $b$.
 Since $M_U$ is everywhere $2$-dimensional, $\dim (M_U\cap V)=2$. Hence, $M_U\cap V$ meets at least one of the sets $G_i$, $i=1,2$.

 Assume first that $M_U\cap V\cap G_1\neq\varnothing$ but $M_U\cap V\cap G_2=\varnothing$.

 Then $M_U\cap W$ meets $G_1$ for every neighborhood $W$ of $b$ with $W\subset V$. Indeed, because $\dim M_U\cap W=2$ and $M_U\cap W\cap G_2=\varnothing$ it follows that $M_U\cap G_1\cap W\neq\varnothing$.
 There consequently is a neighborhood $W$ of $b$ in $G$ such that
\begin{itemize}
\item[(5)] $\overline W\subset V$, $(M_U\cap V)\cap (G_1\setminus\overline W)\neq\varnothing$ and  $M_U\cap G_1\cap W\neq\varnothing$;
\item[(6)] For every $x,y\in W$ there is a homeomorphism $h$ of $G$ supported on $W$ with $h(x)=y$.
\end{itemize}
Finally, choose points $x\in M_U\cap G_1\cap W$ and $y\in W\cap G_2$ and a homeomorphism $h:G\to G$ supported on $W$ with $h(x)=y$. Since $h(z)=z$ for
all points $z\in (M_U\cap V)\cap (G_1\setminus\overline W)$, the set $\widetilde K=h(M_U)$ meets both $G_1$ and $G_2$.
Moreover, the function
$f$ is not extendable over $\bd_F(U\cap F)\cup\widetilde K$ (otherwise $f$ would be extendable over $\bd_F(U\cap F)\cup M_U$). On the other hand, since each of the sets $Q_i=h^{-1}(\widetilde K\cap\overline G_i)$, $i=1,2$, is a proper closed subset of $M_U$, $f$ is extendable over each of the sets
$\bd_F(U\cap F)\cup(\widetilde K\cap\overline G_i)$. Let $\gamma:\bd U\to\mathbb S^1$ be an extension of $f$ (recall that $\dim\bd U\leq 1$ and $\bd_F(U\cap F)$ is a closed subset of $\bd U$, so such $\gamma$ exists). Because $f$ is not extendable over $\bd_F(U\cap F)\cup\widetilde K$, $\gamma$ is not extendable over the set
$K=\bd U\cup\widetilde K\cup C$. Denote
$P_i=C\cup(K\cap\overline G_i)$, $i=1,2$. Obviously, $P_1\cup P_2=K$ and $P_1\cap P_2=C$.
Then for each $i$ we have $P_i\cap\bd U=\{c\}\cup(\bd U\cap\overline G_i)$. So,
the function
$\gamma|(P_i\cap\bd U)$ is extendable over the set
$P_i$ because $\dim C\cup\bd U=1$.
Hence, we can apply
Proposition \ref{promembrane} (with $A=\bd U$) to conclude that there is a continuous function $\beta:C\to\mathbb S^1$ such that $\beta$ is not nullhomotopic, a contradiction.

Assume next that $M_U\cap V$ meets both $G_1$ and $G_2$. We can now proceed as above (considering $M_U$ instead of $\widetilde K$) to obtain the desired contradiction.

%the proof of Lemma~\ref{NIETc}, we can find a limit point $q$ of $V$ in $J$ different from $c$. Let $W$ be a connected open neighborhood of $q$ which misses $\bd U$. Then $V\cup W$ is connected, hence there is by strong local homogeneity, a homeomorphism $\xi$ of $X$ which is supported on $V\cup W$ such that $\xi(p)\in G_2$. Observe that $\xi$ restricts to the identity on $\bd U$. Again, $S=\xi^{-1}(J)\cap M$ is a compact ordered space that separates $M$. Now $S\cap \bd U=\emptyset$ ($c$ can be the only point in that intersection and that point does not belong to $M$), and hence we arrive at a contradiction exactly as in the proof of Lemma~\ref{vrakking}.

\end{document}